\documentclass{amsart}

\usepackage{amsthm}
\usepackage{enumerate}
\usepackage{amsmath}
\usepackage{amssymb}
\usepackage{amsfonts}

\newcommand{\rp}{\mathbb{R}}
\newcommand{\cp}{\mathbb{C}}
\newcommand{\qp}{\mathbb{Q}}
\newcommand{\pp}{\mathbb{P}}
\newcommand{\np}{\mathbb{N}}

\newcommand{\zp}{\mathbb{Z}}

\newcommand{\ep}{\epsilon}

\newcommand{\wt}{\widetilde}

\DeclareMathOperator{\Spec}{Spec}
\DeclareMathOperator{\res}{res}
\DeclareMathOperator{\Pic}{Pic}
\DeclareMathOperator{\Card}{Card}
\DeclareMathOperator{\Vol}{Vol}
\DeclareMathOperator{\Div}{div}
\DeclareMathOperator{\Tr}{Tr}
\DeclareMathOperator{\mult}{mult}
\DeclareMathOperator{\Conv}{Conv}

\newtheorem{lemm}{Lemma}
\newtheorem{coro}{Corollary}
\newtheorem{prop}{Proposition}

\newtheorem{theo}{Theorem}
\newtheorem{rema}{Remark}

\author{Martin Weimann}
\address{Departament Algebra i Geometria, Facultat de Matem\`{a}tiques, Universitat Barcelona
Gran via 585, 08007 Barcelona.}
\email{weimann23@gmail.com}

\title[Osculation...]{Algebraic Osculation And Factorization Of Sparse Polynomials}

\begin{document}


\begin{abstract}
We prove a theorem on algebraic osculation and we apply our result to the Computer Algebra problem of polynomial factorization. 
We consider $X$ a smooth completion of $\cp^2$ and $D$ an effective divisor with support $\partial X=X\setminus \cp^2$. Our main result gives explicit conditions equivalent to that a given Cartier divisor on the subscheme $(|D|,\mathcal{O}_D)$ extends to $X$. These osculation criterions are expressed with residues. We derive from this result a toric Hensel lifting which permits to compute the absolute factorization of a bivariate polynomial by taking in account the geometry of its Newton polytope.  In particular, we reduce the number of possible recombinations when compared to the Galligo-Rupprecht algorithm.  
\end{abstract}

\maketitle


\section{Introduction}

This article is originally motivated by the wellknown Computer Algebra problem of polynomial factorization. We introduce here a new method to compute the absolute factorization of sparse bivariate polynomials, which is based on criterions for algebraic osculation in toric varieties.
 
Since the eightees, several deterministic or probabilistic algorithms have been obtained to compute the irreducible absolute factorization of dense bivariate polynomials defined over a number field $K\subset \cp$. We refer the reader to \cite{GG:gnus} and  \cite{CL:gnus} for  a large overview of the subject.  
In many cases, these algorithms are based on a \textit{Lifting and Recombination} scheme - refered here as LR-algorithms - which detects the irreducible absolute factors of a polynomial $f\in K[t_1,t_2]$ in its formal decomposition in $\bar{K}[[t_1]][t_2]$.
Although this approach \textit{a priori} necessites an exponential number of possible recombinations, people succeded in the last decade to develop LR-algorithms running now in a quasi-optimal polynomial complexity (see for instance \cite{Ch:gnus}, \cite{Gao:gnus}, \cite{CL:gnus} and the reference within). 

In general, a LR-algorithm first performs a generic affine change of coordinates. When $f$ has many zero coefficients in its dense degree $d$ monomial expansion - we say that $f$ is \textit{sparse} - this step loses crucial information. In this article, we propose a new method which avoids this generical choice of coordinates. Roughly speaking, we obtain a toric version of the Hensel lifting process (in the spirit of \cite{Gao:gnus})  which detects and computes the irreducible absolute factors of $f$ by taking in account the geometry of its Newton polytope $N_f$.  This permits in particular to reduce the number of possible recombinations when compared to the Galligo-Rupprecht algorithm \cite{GR:gnus}. 

Let us present our main results.
\vskip2mm
\noindent
\textbf{Algebraic Osculation.} A natural way to take advantage of the Newton polytope information is to embed the complex curve of $f$ in a smooth toric compactification $X$ of the complex plane. For a well chosen $X$, we can recover $N_f$ (up to translation) from the Picard class of the Zariski closure $C\subset X$ of the affine curve of $f$ and we want to use this information. The boundary $\partial X=X\setminus \cp^2$ of $X$ is a normal crossing divisor whose Picard group satisfies
$$
\Pic(\partial X)\simeq \Pic(X).
$$
Thus, it's natural to pay attention to the restriction of $C$ to some effective Cartier divisor $D$ (more precisely, to the subscheme $(|D|,\mathcal{O}_D)$) with support $|\partial X|$. In order to detect the irreducible components of $C$, we need to find conditions for when a Cartier divisor on $D$  extends to $X$.  This is achieved in our main 

\vskip5mm
\noindent
{\bf Theorem 1.} \textit{Let $X$ be a smooth projective compactification of $\cp^2$, whose boundary
$\partial X=X\setminus \cp^2$ is a normal crossing divisor. Let $D$ be an effective Cartier divisor with support $|\partial X|$ and let $\Omega_X^2(D)$ be the sheaf of meromorphic $2$-forms with polar locus bounded by $D$. 
\vskip0mm
There exists a pairing $\langle \cdot,\cdot \rangle$ between the group of Cartier divisors on $D$ and the vector space $H^0(X,\Omega_X^2(D))$ with the property that a Cartier divisor $\gamma$ on $D$ extends to a Cartier divisor $E$ on $X$ if and only if
$$
 \langle \gamma, \Psi \rangle = 0 \quad\forall\,\,\Psi \in H^0(X,\Omega_X^2(D)). 
$$
The divisor $E$ is unique up to rational equivalence.}
\vskip5mm
\noindent

Not surprisingly, we'll construct such a pairing by using Grothendieck residues (see for instance \cite{Grif2:gnus} where the authors study the interplay between residues and zero-dimensional subschemes extension). 
When $X$ is a toric surface we obtain an explicit formula for $\langle \gamma, \Psi \rangle$, generalizing a theorem of Wood \cite{Wood2:gnus}. To prove Theorem $1$, we compute the cohomological obstruction to extend line bundles from $D$ to $X$ and then we use the Dolbeault $\bar{\partial}$-resolution and residue currents to make explicit the conditions.
\vskip3mm
\noindent
\textbf{Application to polynomial factorization.} If two algebraic curves of fixed degree osculate each other with sufficiently big contact orders on some finite subset, they necessarily have a common component. This basic observation permits to derive from Theorem $1$ an algorithm which computes the absolute factorizion of a bivariate polynomial $f$. The polynomial $f$ is assumed to be defined over a subsfield $K\subset \cp$ and we look for its irreducible decomposition over $\cp$. 

Under the assumption $f(0)\ne 0$, we can associate to $f$ a smooth toric completion $X$ of $\cp^2$ whose boundary 
$$
\partial X=D_1+\cdots+D_r
$$ 
is a normal crossing toric divisor and such that the curve $C\subset X$ of $f$ does not contain any torus fixed points of $X$. A Minkowski sum decomposition 
$$
N_f=P+Q
$$
of the Newton polytope of $f$ corresponds to a line bundle decomposition
$$
\mathcal{O}_X(C)\simeq \mathcal{L}_P\otimes \mathcal{L}_Q,
$$
where $\mathcal{L}_P$ and $\mathcal{L}_Q$ are both globally generated with at least one non trivial global section. 
\vskip0mm
\noindent

We prove the following result (Subsection $3.2$)
\vskip5mm
\noindent
{\bf Theorem 2.} \textit{There exists a unique effective divisor $D$ with support $|\partial X|$ and rationally equivalent to $C+\partial X$. Let $\gamma$ be the restriction of $C$ to $D$ and suppose that $N_f=P+Q$. There exists an absolute factor $q$ of $f$ with Newton polytope $Q$ 
if and only if there exists $0\le \gamma'\le \gamma$ such that 
$$
\deg(\gamma'\cdot D_i)=\deg \mathcal{L}_{Q|D_i},\quad i=1,\ldots,r
$$
and so that the osculation conditions hold for the pair $(D,\gamma')$. The factor $q$ is computed from $\gamma'$ by solving a sparse linear system of $2 \Vol(Q)+\deg(\gamma'\cdot \partial X)$ equations and $\Card(Q\cap \zp^2)$ unknowns.}
\vskip5mm
\noindent

When the facet polynomials of $f$ are square free over $\bar{K}$, we can derive from Theorem $2$ a vanishing-sum LR-algorithm (Subsection $3.4$). It first computes the Newton polytope decomposition 
$$
N_f=Q_1+\cdots+Q_s
$$ 
associated to the absolute decomposition of $f$. Then it computes the associated irreducible absolute factorization 
$$
f=q_1\times\cdots\times q_s
$$
with floatting calculous and with a given precision. The numerical part of our algorithm reduces to the absolute factorization of the univariate exterior facet polynomials and the algorithmic complexity depends now on the Newton polytope $N_f$ instead of the degree $d=\deg(f)$. In particular, we fully profit from the combinatoric restrictions imposed by Ostrowski's conditions $N_{pq}=N_p+N_q$ (see \cite{O:gnus}). This permits to reduce the number of possible recombinations when compared to the Galligo-Rupprecht algorithm \cite{GR:gnus}.

Finally, let us mention that Theorem $1$ concerns also non toric $\cp^2$-completions. In theory, this permits to exploit the information given by the \textit{non toric} singularities of $C$ along the boundary of $X$ when $f$ has non reduced facet polynomials (see Subsection $3.6$).

A formal study of algorithmic complexity, as well as questions of using non toric singularities information will be explored in a further work. 
\vskip4mm
\noindent
{\bf Related results.} Our method is inspired by an algorithm presented in \cite{EGW:gnus} that uses generical toric interpolation criterions \cite{W:gnus} in an open neighborhood of $\partial X$ (see Subsection $3.5$). By using combinatorics tools, the authors in \cite{Gao:gnus} obtain a comparable Hensel lifting process which there too takes in account the geometry of the Newton polytope. Finally, let us mention \cite{AKS:gnus}, where the authors reduce the factorization of a sparse polynomial to smaller dense factorizations. 
\vskip4mm
\noindent
{\bf Organization.} The article is organized as follow. Section $2$ is devoted to algebraic osculation. We introduce the problem in Subsection $2.1$ and we construct the residue pairing in Subsection $2.2$. We enounce precisely Theorem $1$ in Subsection $2.3$ and we give the proof in Subsection $2.4$. We give an explicit formula for the osculation criterions when $X$ is a toric surface in Subsection $2.5$. 
In Section $3$, we pay attention to polynomial factorization. We prove Theorem $2$ in Subsections $3.2$ and $3.3$ and we develop the sketch of a sparse polynomial factorization algorithm in Subsection $3.4$. We compare the underlying algortihmic complexity with related results in Subsections $3.5$ and discuss non toric information in Subsection $3.6$. We conclude in the last Subsection $3.7$.

\vskip4mm
\noindent
{\bf Aknowledgment.} We would like to thanks Michel Brion, St\'ephane Druel, Jos\'e Ignacio Burgos and Martin Sombra for their disponibility and helpfull comments. We thanks Mohamed Elkadi and Andr\'e Galligo who suggested us to pay attention to the interplay between toric geometry and sparse polynomial factorization. 

\section{Algebraic osculation}

\subsection{Notations and motivation} In all the sequel, $(X,\mathcal{O}_X)$ designs a smooth projective surface where   
$$
X=X_0\sqcup |\partial X|
$$
is the disjoint union of an affine surface $X_0\simeq \cp^2$ with the support of a simple normal crossing divisor 
$$
\partial X=D_1+\cdots+D_r.
$$
We say that $X$ is a completion of the complex plane with boundary $\partial X$. 
\vskip2mm
\noindent

An \textit{osculation data} on the boundary of $X$ is a pair 
$
(D,\gamma)
$ where
$$
D=(k_1+1)D_1+\cdots +(k_r+1)D_r
$$
is an effective divisor with support $|D|=|\partial X|$ and $\gamma$ is a Cartier divisor on the subscheme $(|D|,\mathcal{O}_D)$. By abuse of language, we will write $D=(|D|,\mathcal{O}_D)$.

An \textit{osculating divisor} for $(D,\gamma)$ is a Cartier divisor $E$ on $X$ which restricts to $\gamma$ on $D$. That is
$$
i^*(E)=\gamma,
$$
where $i:D\rightarrow X$ is the inclusion map. In other words, we are looking for a divisor $E$ with prescribed restriction to the $k_i^{th}$-infinitesimal neighborhood of $D_i$. In general, such an osculating divisor does not exist and we are interested here to determine the necessary extra conditions.  

\vskip2mm
\noindent

We say that $\gamma$ has support $|\Gamma|$, where $\Gamma$ designs the zero-cycle $\gamma\cdot \partial X$. Thus, $\gamma$ can be uniquely written as a finite sum
$$
\gamma=\sum_{p\in|\Gamma|} \gamma_p,
$$
each $\gamma_p$ being the restriction to $D$ of a germ of an analytic \textit{cycle} of $X$
$$\wt{\gamma}_p=\Div(f_p)$$ 
at $p$. When $\wt{\gamma}_p$ is smooth and intersects transversally $\partial X$, a curve restricts to $\gamma_p$ at $p$ if and only if it has contact order at least $k_p$ with $\wt{\gamma}_p$, where $k_p+1$ is the multiplicity of $D$ at $p$. This observation motivates the terminology of algebraic osculation.


\subsection{Residues} It's a classical fact that Grothendieck residues play a crucial role in osculation and interpolation problems. Let us mention for instance \cite{Grif3:gnus}, \cite{HP:gnus}, \cite{W2:gnus} and \cite{W:gnus} for interpolation results and \cite{Grif2:gnus} for the interplay between residues and subscheme extensions. Not surprisingly, residues will appear here too. 

Let $\Omega^2_X$ be the canonical bundle of $X$. We identify the line bundle $\Omega_X^2(D)$ with the sheaf of meromorphic forms with polar locus bounded by $D$. Thus, any global section $\Psi$ of $\Omega_X^2(D)$ restricts to a closed $2$-form on $X_0$. Since $X_0\simeq \cp^2$ is simply connected, there exists a rational $1$-form $\psi$ on $X$ such that 
$$
d\psi_{|X_0} = \Psi_{|X_0}.
$$

For $p\in |\Gamma|$, we denote by $\psi_p$ the germ of $\psi$ in the chosen local coordinates. Let $h_p$ be a local equation of $D$ at $p$. Thus $h_p\psi_p$ is holomorphic at $p$. Suppose for a while that $f_p$ is holomorphic and irreducible. Then, following \cite{Grif:gnus}, we define the Grothendieck residue at $p$ of the germ of meromorphic $2$-form $df_p\land \psi_p/f_p$ as
\begin{eqnarray}
\res_p\,\Big[\frac{df_p}{f_p}\land \psi_p\Big]:=\lim_{\underline{\ep}\rightarrow 0} \frac{1}{(2i\pi)^2} \int_{u_p(\ep)} \frac{df_p}{f_p}\land \psi_p,
\end{eqnarray}
where $u_p(\ep)=\{x \,\,{\rm close\,\, to}\, p,\,\,\,|f_p(x)|=\ep_1, |h_p|=\ep_2\}$. 

This definition does not depend on the choice of local coordinates \cite{Grif:gnus}. By Stokes Theorem, it only depends on $d\psi=\Psi$. Moreover, the local duality Theorem \cite{Grif:gnus} implies that $(1)$ only depends on $f_p$ modulo $(h_p)$. That is, $(1)$ depends on $\gamma_p$ and not on the chosen lifting $\wt{\gamma}_p=\{f_p=0\}$. By linearity, we can extend $(1)$ to any germ of analytic cycle $\wt{\gamma}_p=\Div(f_p)$ and it follows finally that $(1)$ defines a bilinear operator 
$$
\langle \gamma, \Psi \rangle_p:= \res_p\,\Big[\frac{df_p}{f_p}\land \psi_p\Big]
$$
between the group of Cartier divisor of $D$ and the $\cp$-vector space $H^0(X, \Omega_X^2(D))$. We refer to the proof of Theorem $1$ (Subsection $2.4$) for more details. 

\vskip4mm

\subsection{Criterions for algebraic osculation}

We keep previous notations. Our main result is the following 
\vskip4mm
\begin{theo}
Let $X$ be a smooth projective completion of $\cp^2$ with a normal crossing boundary $\partial X$ and consider an osculating data $(D,\gamma)$ on $\partial X$. 

1. There exists a Cartier divisor $E$ on $X$ which restricts to $\gamma$ on $D$ if and only if
\begin{eqnarray}
\sum_{p\in |\Gamma|}  \langle \gamma, \Psi \rangle_p = 0  \,\,\,for\,\,\, all\,\,\, \Psi \in H^0(X,\Omega_X^2(D)).
\end{eqnarray}
The divisor $E$ is unique up to rational equivalence. 

2. If moreover $\gamma$ is effective and
$$
H^1(X,\mathcal{O}_X(E-D))=0,
$$
then there exists an effective divisor of $X$ which restricts to $\gamma$ on $D$.
\end{theo}

\vskip4mm
\noindent

The necessity of $(2)$ follows from the Residue Theorem \cite{Grif:gnus}. The difficult part consists to show that these conditions are also sufficient. The proof will be given in the next Subsection $2.4$. Let us first illustrate  Theorem $1$ on a simple example.
\vskip2mm
\noindent
{\it Example 1 (the Reiss relation).} Let $X=\pp^2$ and consider a finite collection of $d>0$ \textit{smooth} analytic germs $\wt{\gamma}_p$ transversal to a line $L\subset \pp^2$. Suppose that we look for an algebraic curve $C\subset \pp^2$ of degree $d$ which osculates each germ with a contact order $\ge 2$.  This problem leads to the osculation data $(D,\gamma)$, where $D=3L$ and $\gamma$ is the restriction to $D$ of $\sum_{p\in |\Gamma|}\wt{\gamma}_p$. 
There is an isomorphism
$$
H^0(\pp^2,\Omega_{\pp^2}^2(3L))\simeq H^0(\pp^2,\mathcal{O}_{\pp^2})\simeq \cp
$$
and we can check that a generator is given by the form $\Psi$ whose restriction to $\cp^2=\pp^2\setminus L$ is equal to
$
\Psi_{|\cp^2}=dt_1\land dt_2
$
in affine coordinates $(t_1,t_2)$. Thus we can choose $\psi_{|\cp^2}=t_1dt_2$. Letting $t_1=T_1/T_0$ and $t_2=T_2/T_0$ we obtain  
$$
\psi=\frac{T_1(T_0dT_2-T_2dT_0)}{T_0^3}
$$
in the $\pp^2$ homogeneous coordinates $[T_0:T_1:T_2]$. Up to an affine change of coordinates, we can suppose that $\gamma$ is supported in the affine chart $T_2 \ne 0$. In the new affine coordinates $x=T_0/T_2$ and $y=T_1/T_2$, the line $L$ has local equation $x=0$ and $\psi=-ydx/x^3$. Moreover, we can choose a Weierstrass equation 
$$
\wt{\gamma}_p=y-\phi_p(x)
$$
for the smooth germ $\wt{\gamma}_p$, where $\phi_p\in\cp\{x\}$. By Cauchy formula, we obtain
$$
\res_p \Big[\frac{ydx\land d(y-\phi_p)}{x^3(y-\phi_p)}\Big]= \res_0 \Big[\phi_p(x) \frac{dx}{x^3}\Big]= \frac{1}{2}\phi_p''(0),
$$
where $\res_0$ is a univariate residue and $\phi_p''$ is the second derivative of $\phi_p$. Finally, $(2)$ is here equivalent to that
\begin{eqnarray}
\sum_{p\in |\Gamma|} \phi_p''(0)=0.
\end{eqnarray}
For degree reasons, any osculating divisor is rationally equivalent to $dL$. Since 
$
H^1(\mathcal{O}_{\pp^2}(d-3))=0,
$
the relation $(3)$ is finally equivalent to that there exists an osculating \textit{curve} $C$ for $(D,\gamma)$. 

It's easy to see directly the necessity of $(3)$. An osculating curve is given by a degree $d$ polynomial $C=\{f(x,y)=0\}$ that can be factorized
$$
f(x,y)=\prod_{p\in|\Gamma|} (y-u_p(x))
$$
in $\cp\{x\}[y]$. Since $deg(f)=d$, the sum $\sum_{p\in |\Gamma|}u_p(x)$ is a degree $1$ polynomial and the relation
$$
\sum_{p\in |\Gamma|} u_p''(0)=0
$$
holds. If $C$ osculates $\wt{\gamma}_p$ with contact oder $2$, then $\phi_p$ and $u_p$ have the same Taylor expansion up to order $2$, which directly shows necessity of $(3)$. When we express the second derivative of $u_p$ in terms of the partial derivative of $f$, we recover the classical Reiss relation \cite{Green:gnus}. This result is obtained by Griffiths-Harris in \cite{Grif:gnus}, Chapter $6$. 

\subsection{Proof of Theorem $1$}

The Cartier divisor $\gamma$ corresponds to a line bundle $\mathcal{L}\in \Pic(D)$ over $D$ together with a global meromorphic section $f$. An osculating divisor for $(D,\gamma)$ corresponds to a  line bundle 
$\wt{\mathcal{L}}\in \Pic(X)$ together with a global meromorphic section $\wt{f}$ which restrict respectively to $\mathcal{L}$ and $f$ on $D$.

\vskip2mm

The following lemma gives the cohomological obstruction for the extension of $\mathcal{L}$. All sheaves are considered here as analytic sheaves and any sheaf on a subscheme $Y\subset X$ is implicitly considered as a sheaf on $X$ by zero extension.

\begin{lemm}
There is a decomposition of $\Pic(D)$ in a direct sum
$$
\Pic(D)=\Pic(X)\oplus H^1(D,\mathcal{O}_D).
$$
\end{lemm}

\begin{proof}

The classical exponential exact sequence exists for any curve of $X$ (reduced or not, see \cite{BHPV:gnus} p.63). We obtain the commutative diagram
\begin{equation}
\begin{array}{ccccccccc}
0 &\longrightarrow& \zp_{X} &\longrightarrow & \mathcal{O}_X &\stackrel{\exp(2i\pi\cdot)}{\longrightarrow} & \mathcal{O}_X^* &\longrightarrow& 0 \\
& &\downarrow&  &\downarrow&  &\downarrow  \\
0 &\longrightarrow& \zp_{D} &\longrightarrow & \mathcal{O}_D &\longrightarrow &\mathcal{O}_D^*&\longrightarrow &0\\
\end{array}
\end{equation}
where $\zp_{D}\subset \mathcal{O}_D$ is the subsheaf of $\zp$-valued functions (so that $\zp_{D}\simeq \zp_{|D|}$) and vertical arrows are surjective restriction maps. Since $X$ is rational, $H^i(X,\mathcal{O}_X)=0$ for $i>0$. Using the associated long exact cohomological sequences, we obtain the commutative diagram
$$
\begin{array}{ccccccccc}
& & 0 & \rightarrow & \Pic(X) & \stackrel{\delta_X}{\rightarrow} & H^2(X,\zp_{X})& \rightarrow & 0    \\
& & \downarrow &  & \downarrow r         &             &   \downarrow r'  &  & \downarrow     \\
H^1(D,\zp_{D})& \rightarrow & H^1(D,\mathcal{O}_D) & \stackrel{e}{\rightarrow} & \Pic(D) & \stackrel{\delta_D}{\rightarrow} & H^2(D,\zp_{D}) & \rightarrow & H^2(D,\mathcal{O}_D)     . \\
\end{array}
$$
Here $r$, $r'$ are restriction maps, $\delta_D$, $\delta_X$ are the coboundary maps corresponding to Chern classes on $D$ and $X$ (see \cite{BHPV:gnus}, Ch. $1$ for the non reduced case) and $e$ is induced by the exponential map. 

We claim that $r'$ is an isomorphism. Let $j:X_0\rightarrow X$ be the inclusion map of $X_0=X\setminus |D|$. The short exact sequence 
\begin{equation}
0\rightarrow j_!(\zp_{X_0})\rightarrow \zp_X \rightarrow \zp_{|D|}\rightarrow 0
\end{equation}
gives rise to the long exact cohomological sequence
$$
H_c^2(X_0)\rightarrow H^2(X,\zp_X) \rightarrow H^2(|D|,\zp_{|D|})\rightarrow H_c^3(X_0)
$$ 
where $H_c^*(X_0)$ is the compact support cohomology of $X_0$. The cohomology $H_c^*(X_0)$ is dual to the singular homology $H_{*}(X_0)$ which vanishes in degree $1,2,3$ since $X_0\simeq \cp^2$, and the claim follows. Thus 
$$
\delta_D\circ r = r' \circ \delta_X 
$$
is an isomorphism, and $\Pic(X)$ is a direct summand of $\Pic(D)$. The exponential cohomological sequence for $X$ is exact in degree $0$ so that $H^1(X,\zp_X)\rightarrow  H^1(X,\mathcal{O}_{X})$ is injective and $H^1(X,\zp_X)=0$. We have just seen that $H^2_c(X_0)=0$, and finally $(5)$ implies $H^1(|D|,\zp_{|D|})=0$. It follows that $ker(\delta_D)\simeq H^1(D,\mathcal{O}_D)$. \end{proof}

\begin{coro}
There is an isomorphism $\Pic(X)\simeq \Pic(\partial X)$.
\end{coro}

\begin{proof}
By the previous lemma, it's enough to show that $H^1(D_{red},\mathcal{O}_{D_{red}})=0$ where $\partial X=D_{red}$ is the reduced part of $D$. Since $H^1(Z,\zp_{Z})=0$ for a finite set $Z$, there is a surjection
$$
H^1(|D|,\zp_{|D|})\rightarrow \bigoplus_{i=1}^r H^1(D_i,\zp_{D_i}).
$$ 
Thus $H^1(D_i,\zp_{D_i})=0$ for all $i$ and each $D_i$ is a \textit{rational} curve. In particular, $H^1(\mathcal{O}_{D_{i}})=0$. Moreover, $H_c^1(X_0)=0$ so that $H^0(X,\zp_{X})\rightarrow H^0(|D|,\zp_{|D|})$ is surjective and $|D|$ is connected. This shows that $|D|$ is a connected and simply connected tree of rational curves.  There thus exists $i$ so that $D_{red}=D_i+D'$, where $|D_i|\cap |D'|$ is a point and $D'$ remains a connected and simply connected tree. By induction on the number of irreducible branches of $D_{red}$, it's enough to show that $H^1(\mathcal{O}_{D_{red}})\simeq H^1(\mathcal{O}_{D'})$. The short exact sequence
$$
0\rightarrow \mathcal{O}_{D_{red}}\rightarrow \mathcal{O}_{D_i}\oplus \mathcal{O}_{D'} \rightarrow \mathcal{O}_{D_i\cdot D^{'} } \rightarrow 0
$$
gives the long exact sequence in cohomology 
\begin{eqnarray*}
0 &\rightarrow& H^0(\mathcal{O}_{D_{red}})\rightarrow H^0(\mathcal{O}_{D_i})\oplus H^0(\mathcal{O}_{ D'}) \rightarrow H^0(\mathcal{O}_{D_i \cdot D'}) \\ 
&\rightarrow& H^1(\mathcal{O}_{D_{red}})\stackrel{r}{\rightarrow} H^1(\mathcal{O}_{D_i})\oplus H^1(D',\mathcal{O}_{D'})\rightarrow H^1(\mathcal{O}_{D_i \cdot D'}).
\end{eqnarray*}
By assumption, $D_{red}$ is a \textit{normal crossing divisor} and $D_i \cdot D'=\{pt\}$ as a \textit{reduced} subscheme. The diagram then begins by $0 \rightarrow \cp \rightarrow \cp\oplus \cp \rightarrow \cp$, which forces $r$ to be injective. Since $H^1(\mathcal{O}_{D_i})=H^1(\mathcal{O}_{\{pt\}})=0$, this gives $H^1(\mathcal{O}_{D_{red}})\simeq H^1(\mathcal{O}_{D'})$. 
\end{proof}
By the previous corollary, there exists a unique line bundle $\wt{\mathcal{L}}\in Pic(X)$ so that $\wt{\mathcal{L}}_{|\partial X}\simeq\mathcal{L}_{|\partial X}$ and $\mathcal{L}$ lifts to $X$ if and only if 
$$
\mathcal{L}_0:=\mathcal{L}\otimes\wt{\mathcal{L}}_{|D}^{-1}\simeq\mathcal{O}_D.
$$
Since $\delta_D(\mathcal{L}_0)=0$ (by construction), there exists $\beta\in H^1(D,\mathcal{O}_D)$ with $e(\beta)=\mathcal{L}_0$. Moreover $e$ is injective and 
$$ 
\mathcal{L}_0\simeq \mathcal{O}_D \iff \beta=0.
$$
There remains to make explicit such a condition. We first use Cech cohomology. Suppose that $\mathcal{L}_0$ is given by the $1$-cocycle class $g=\{g_{UV}\}\in H^1(D,\mathcal{O}_D^*)$, relative to some open covering $\mathcal{U}$ of $X$. 
We can suppose $\mathcal{U}$ fine enough to ensure a logarithmic determination 
$$
\wt{h}_{UV}:=\frac{1}{2i\pi} \log(\wt{g}_{UV})\in \mathcal{O}_X(U\cap V),
$$
where $\wt{g}_{UV}\in\mathcal{O}^*_X(U\cap V)$ is some local lifting of $g_{UV}$. Since $\delta_D(g)=0$, the classes $h_{UV}\in \mathcal{O}_D(U\cap V)$ of $\wt{h}_{UV}$  define a $1$-cocycle class $\{h_{UV}\}\in H^1(D,\mathcal{O}_D)$ which represents $\beta$ (this definition does not depend on the choice of the liftings). Since $X$ is rational, the structural sequence for $D$
$$
0 \longrightarrow \mathcal{O}_X(-D) \longrightarrow \mathcal{O}_X  \longrightarrow  \mathcal{O}_D \longrightarrow 0 
$$
gives rise to a coboundary \textit{isomorphism} 
$
\delta : H^1(D,\mathcal{O}_D)\rightarrow H^2(X,\mathcal{O}_X(-D))
$
while Serre Duality gives a non degenerate pairing 
$$
H^2(X,\mathcal{O}_X(-D))\otimes H^0(X,\Omega_X^2(D)) \stackrel{(\cdot,\cdot)}{\longrightarrow} H^2(X,\Omega^2_X)\stackrel{\Tr}{\simeq} \cp,
$$
where $\Tr$ is the trace map \cite{BHPV:gnus}. We identify $\mathcal{O}_X(-D)$ with the sheaf of functions vanishing on $D$ and $\Omega_X^2(D)$ with the sheaf of meromorphic $2$-forms with polar locus bounded by $D$. The Serre pairing 
$
(\delta(\beta),\Psi)
$
with $\Psi\in H^0(X,\Omega^2_X(D))$ is represented by the $2$-cocyle class
$$
\zeta_{\Psi}:=\Big\{(\wt{h}_{UV}+\wt{h}_{VW}+\wt{h}_{WU})\Psi_{|U\cap V\cap W}\Big\}\in H^2(X,\Omega^2_X)
$$
and $\beta=0$ if and only if
$$
\Tr(\zeta_{\Psi})=0\quad\forall\,\,\Psi\in H^0(X,\Omega^2_X(D)).
$$
To make explicit the complex numbers $\Tr(\zeta_{\Psi})$, we use the Dolbeault $\bar{\partial}$-resolution of $\Omega_X^2$. We denote by $\mathcal{D}_X^{(p,q)}$ the sheaf of germs of $(p,q)$-currents on $X$. Let  $\psi$ be a germ of meromorphic $q$-form  at $p\in X$. We recall for convenience that the \textit{principal value} current $[\psi]\in \mathcal{D}_{X,p}^{(q,0)}$ and the \textit{residue current} $\bar{\partial}[\psi]\in \mathcal{D}_{X,p}^{(q,1)}$ have the Cauchy integral representations
$$
\langle [\psi], \theta\rangle :=\lim_{\ep\rightarrow 0} \frac{1}{2i\pi} \int_{U\cap\{|u|>\ep\}} \psi\land \theta
$$
and
$$
\langle \bar{\partial}[\psi], \theta \rangle  := \lim_{\ep\rightarrow 0} \frac{1}{2i\pi} \int_{U\cap\{|u|=\ep\}} \psi\land\theta,
$$
where $\theta$ is some test-form with appropriate bidegree and $u=0$ is a local equation for the polar divisor of  $\psi$ in a small neighborhood $U$ of $p$ (see \cite{T:gnus} for instance). The Dolbeault resolution of $\Omega_X^2$ is  given by the exact complex of sheaves
$$
0\longrightarrow \Omega_X^2 \stackrel{[\cdot]}{\longrightarrow} \mathcal{D}_X^{(2,0)} \stackrel{\bar{\partial}}{\longrightarrow}\mathcal{D}_X^{(2,1)} \stackrel{\bar{\partial}}{\longrightarrow}\mathcal{D}_X^{(2,2)} \longrightarrow 0.
$$
This sequence breaks in the two short exact sequences
$$
0\rightarrow \Omega_X^2 \rightarrow \mathcal{D}_X^{(2,0)} \stackrel{\bar{\partial}}{\rightarrow}\mathcal{Z}_X^{(2,1)}\rightarrow 0
\quad {\rm 
and}\quad
0\rightarrow \mathcal{Z}_X^{(2,1)}\rightarrow \mathcal{D}_X^{(2,1)} \stackrel{\bar{\partial}}{\rightarrow}\mathcal{D}_X^{(2,2)} \rightarrow 0,
$$ 
where $\mathcal{Z}_X^{(2,1)}\subset \mathcal{D}_X^{(2,1)}$ is the subsheaf of $\bar{\partial}$-closed $(2,1)$-currents. Since the sheaves $\mathcal{D}_X^{(p,q)}$ are fine, we obtain the two coboundary \textit{isomorphisms} 
$$
\frac{H^0(X,\mathcal{D}_X^{(2,2)})}{\bar{\partial}H^0(X,\mathcal{D}_X^{(2,1)})}\stackrel{\delta_1}{\longrightarrow} H^1(X,\mathcal{Z}_X^{(2,1)})\quad {\rm and }\quad H^1(X,\mathcal{Z}_X^{(2,1)})\stackrel{\delta_2}{\longrightarrow}H^2(X,\Omega_X^2).
$$
For any $\Psi\in H^0(X,\Omega^2_X(D))$, there thus exists a global current $T_{\Psi}\in H^0(X,\mathcal{D}_X^{(2,2)})$ whose class $[T_{\Psi}]$ modulo $\bar{\partial}$ is unique solution to $\zeta_{\Psi}=\delta_2\circ \delta_1([T_{\Psi}])$, and we have equality
$$
\Tr(\zeta_{\Psi})=\langle T_{\Psi},1\rangle.
$$
To make explicit $T_{\Psi}$, we need the following lemma.
\begin{lemm}
1. There exists a Cartier divisor  $\wt{\gamma}$ on an open neighborhood $B$ of $|D|$ which restricts to $\gamma$ on $D$.

2. There exists a Cartier divisor $E$ on $X$ which restricts to $\Gamma$ on $D_{red}$ and there is an isomorphism $\wt{\mathcal{L}}\simeq\mathcal{O}_X(E)$.
\end{lemm}

\begin{proof}
1. By mean of a partition of unity, we can construct a $\mathcal{C}^{\infty}$ function $u$ on $X$ which vanishes exactly on $|D|$, giving an open tubular neighborhood of $D$
$$
B_{\ep}=\{|u|<\ep\}.
$$
Let $f=\{f_U\}$ be a meromorphic section of $\mathcal{L}$ with Cartier divisor $\gamma$, and consider some local meromorphic liftings $\wt{f}_U$. We can choose $\mathcal{V}$ a sufficiently fine covering of $B_{\ep}$ so that $\wt{f}_U\in\mathcal{O}_{X}(U)^*$ for $U\cap |\Gamma|=\emptyset $ and so that $U\cap V\cap |\Gamma|=\emptyset$ for distinct $U, V\in \mathcal{V}$. If now $U$ intersects $|\Gamma|$, we can choose $\ep'<\ep$ small enough so that $\wt{f}_U$ has no poles nor zeroes on $U\cap V \cap B_{\ep'}$. In such a way, $\wt{f}_U/\wt{f}_V$ is invertible on $U\cap V\cap B_{\ep'}$, giving a Cartier divisor $\wt{\gamma}$ on $B_{\ep'}$ which restricts to $\gamma$ on $D$.

2. Let $\mathcal{F}=\mathcal{O}_X(F)$ be some very ample line bundle on $X$ with $F$ an effective divisor which intersects properly $|D_{red}|$. The isomorphisms 
$$
\mathcal{O}_{D_{red}}(\Gamma)\simeq \mathcal{L}_{|D_{red}}\simeq\wt{\mathcal{L}}_{|D_{red}}
$$
combined with the structural sequence for $D_{red}$ gives the exact sequence
$$
0\rightarrow \wt{\mathcal{L}}\otimes\mathcal{O}_X(nF-D_{red})\rightarrow \wt{\mathcal{L}}\otimes\mathcal{O}_X(nF)\rightarrow \mathcal{O}_{D_{red}}(\Gamma+n\Gamma_F)\rightarrow 0
$$
for any integer $n$, where $\Gamma_F=F \cdot D_{red}$. For $n$ big enough, we have 
$$
H^1(\wt{\mathcal{L}}\otimes \mathcal{O}_X(nF-D_{red}))=0
$$ 
by Serre Vanishing Theorem so that $\Gamma+n\Gamma_F$ lifts to some Cartier divisor $G$ on $X$. Then, the Cartier divisor $E=G-nF$ restricts to $\Gamma$ on $D_{red}$. By construction $\mathcal{O}_X(E)$ and $\mathcal{L}=\mathcal{O}_D(\gamma)$ have the same restriction to $D_{red}=\partial X$, and Proposition $1$  implies $\wt{\mathcal{L}}\simeq\mathcal{O}_X(E)$.
\end{proof}
Let $B, \wt{\gamma}$ and $E$ as in Lemma $2$, with $\mathcal{V}$ the associated open covering of $B$. Since 
$
\mathcal{O}_{B}(\wt{\gamma})\otimes \mathcal{O}_{B}(-E)
$
restricts to $\mathcal{L}_{0}$ on $D$, we can choose the liftings $\wt{g}_{UV}=m_U/m_V$, where $m=\{m_U\}_{U\in\mathcal{V}}$ is the global meromorphic section of $\mathcal{O}_{B}(\wt{\gamma})\otimes \mathcal{O}_{B}(-E)$ with Cartier divisor 
$$
\Div(m)=\wt{\gamma} - E_{|B}
$$ 
on $B$. For $B$ small enough, $m_{U|U\cap V}\in\mathcal{O}_B^*(U\cap V)$ for distinct $U, V\in\mathcal{V}$ and there exists a logarithmic determination
$$
\log (m_{U|U\cap V}):=\sum_{n=1}^{\infty}\frac{1}{n}\big(1-m_{U|U\cap V}\big)^n.
$$ 
Since $\Div(m)\cdot D_{red}=0$, the restriction of $m$ to $D_{red}$ is a constant, that we can suppose to be $1$. There thus exists $n_0\in\np$ such that the meromorphic function $(1-m_U)^n$ vanishes on the divisor $D\cap U$ for all $n\ge n_0$ and all $U\in\mathcal{V}$. 

Let $\Psi\in H^0(X,\Omega^2_X(D))$. We can multiply a principal value current (or a residue current) with a residue current as soon as their supports have a proper intersection. Then, by the Duality Theorem \cite{T:gnus}, we have
$$
\big[\big(1-m_U\big)^n\big]\bar{\partial}\big[\Psi_U\big]=0\quad\forall\,n\ge n_0
$$
so that the principal value currents 
$$
S_U:=\sum_{n=1}^{n_0}\frac{1}{n}\Big[\big(1-m_U\big)^n\Big]
$$
satisfy 
\begin{eqnarray}
(S_U-S_V)\bar{\partial}\big[\Psi_{U\cap V}\big]=\big[\log(\wt{g}_{U\cap V})\big]\bar{\partial}\big[\Psi_{U\cap V}\big]
\end{eqnarray}
for all $U, V\in\mathcal{V}$. We obtain in such a way a global $(2,2)$-current $T_{B}\in \Gamma(B,\mathcal{D}_B^{(2,2)})$ on $B$, locally given by
$$
T_{U}=\bar{\partial}S_U\land \bar{\partial}[\Psi_U]
$$
for all $U\in \mathcal{V}$. Since $\Psi$ is holomorphic outside $|D|$, we have $T_U=0$ when $U\cap |D|=\emptyset$ and we can extend $T_B$ by zero to a global current $T$ on $X$. Using $(6)$ and the definition of the coboundary maps, we check easily that $T$ is solution to
$$
\delta_2\circ \delta_1(T)=\zeta_{\Psi}.
$$

Let now $\psi$ be some rational $1$-form on $X$ so that
$d\psi_{|X_0}=\Psi_{|X_0}$. For all $U\in\mathcal{V}$, we define the local currents
$$
R_{U}:=\Big(d S_U-\Big[\frac{dm_U}{m_U}\Big]\Big) \land \bar{\partial}\big[\psi_U\big]\quad{\rm and}\quad T'_U=\bar{\partial} \Big[\frac{dm_U}{m_U}\Big] \land \bar{\partial}\big[\psi_U\big].
$$
From $(6)$ we deduce equalities $R_{U|U\cap V}=R_{V|U\cap V}$, and the $R_U$'s define a global $(2,1)$-current $R_B$ on $B$. By assumption, $m_U/m_V$ is invertible on $U\cap V$ so that $T'_{U|U\cap V}=T'_{V|U\cap V}$, giving a global $(2,2)$-current $T'_B$ on $B$. Both currents $R_B$ and $T'_B$ are supported on $|D|\subset B$ and can be extended by zero to global currents $R$ and $T'$ on $X$. 

Since $T_U$ has support a finite set, Stokes Theorem gives equalities
\begin{eqnarray*}
\langle T_{U}, 1\rangle &=& \langle \bar{\partial}S_U\land \bar{\partial}[\Psi_U], 1\rangle \\
&=& \langle \bar{\partial}S_U\land \bar{\partial}d[\psi_U],1\rangle\\
&=& \langle \bar{\partial} (dS_U)\land \bar{\partial}[\psi_U],1\rangle = \langle T'_{U}+\bar{\partial}{R_U}, 1\rangle.
\end{eqnarray*}
Thus, $\langle T, 1 \rangle=  \langle T'+\bar{\partial}R, 1\rangle = \langle T', 1 \rangle$.
Since $\Div(m)=\wt{\gamma}-E_{|B}$, the Lelong-Poincar\'e equation gives equality
$$
\langle T', 1\rangle= \langle T'_B, 1\rangle = \langle [\wt{\gamma}]\land\bar{\partial}[\psi]_{|B}, 1\rangle-\langle [E]_{|B}\land\bar{\partial}[\psi]_{|B}, 1\rangle,
$$
where $[\cdot]$ designs here the integration current associated to analytic cycles. 
The current $[E]_{|B}\land\bar{\partial}[\psi]_{|B}$ is supported on the compact subset $|\Gamma|\subset B$,  and the integration current is $\bar{\partial}$-closed. We deduce
$$
\langle [E]_{|B}\land\bar{\partial}[\psi]_{|B}, 1\rangle=\langle [E]\land\bar{\partial}[\psi], 1\rangle=\langle \bar{\partial}([E]\land[\psi]), 1\rangle=0.
$$
The analytic set $B\cap \wt{\gamma}$ is a disjoint union of analytic cycles  $\wt{\gamma}_p=\{f_p=0\}$. If $\psi_p$ is the germ of $\psi$ at $p$ in the chosen local coordinates, we obtain finally
\begin{eqnarray*}
\Tr(\zeta_{\Psi}) &=& \sum_{p\in|\Gamma|}\langle [\wt{\gamma}_p]\land\bar{\partial}[\psi_p],1 \rangle\\
&=& \sum_{p\in|\Gamma|} \res_p\big(\frac{df_p}{f_p}\land \psi_p\big)= \sum_{p\in|\Gamma|}\langle \gamma,\Psi\rangle_p,
\end{eqnarray*}
and it follows that $(2)$ is equivalent to that $\mathcal{L}=\mathcal{O}_D(\gamma)$ extends to $\wt{\mathcal{L}}\in \Pic(X)$. Note that the previous expression only depends on the line bundle $\mathcal{O}_D(\gamma)$ and on $\Psi$ by construction.
If $\mathcal{L}$ extends to $\wt{\mathcal{L}}$, we can use Vanishing Serre Theorem as in Lemma $2$ (with $(D,\gamma)$ instead of $(D_{red},\Gamma)$) and show that $\wt{\mathcal{L}}\simeq \mathcal{O}_X(E)$ for some Cartier divisor $E$ on $X$ which restricts to $\gamma$ on $D$. The lifting bundle $\wt{\mathcal{L}}$ being unique (up to isomorphism), the osculating divisor $E$ is unique up to rational equivalence. This ends the proof of the first point. 

If conditions $(2)$ hold, then $\mathcal{L}\simeq \mathcal{O}_D(E)$ for some Cartier divisor $E$ on $X$. By tensoring the structural sequence of $D$ with $\wt{\mathcal{L}}=\mathcal{O}_X(E)$, we obtain the short exact sequence
$$
0\rightarrow \mathcal{O}_X(E-D)\rightarrow \mathcal{O}_X(E)\rightarrow \mathcal{O}_D(E)\rightarrow 0.
$$
If $\gamma$ is effective and $H^1(X,\mathcal{O}_X(E-D))=0$, the global section $f\in H^0(D,\mathcal{O}_D(E))$ with zero divisor $\gamma$ automatically lifts to $\wt{f}\in H^0(X,\mathcal{O}_X(E))$ and $C:=\Div_0(\wt{f})$ is an \textit{effective} osculating divisor for $(D,\gamma)$. This ends the proof of Theorem $1$. $\hfill{\square}$

\vskip2mm
\noindent

\subsection{An explicit formula in the toric case} We show now that we can make explicit the conditions $(2)$ in the case of a toric variety $X$. We refer the reader to \cite{F:gnus} and \cite{Danilov:gnus} for an introduction to toric geometry.

\subsubsection{Preliminaries} Suppose that $X$ is a \textit{toric} surface containing $X_0$ as an union of orbits. Then, $X$ has toric divisors $D_0,\ldots,D_{r+1}$ where 
$$
\partial X=D_1+\cdots +D_r
$$
is the boundary of $X$ and $D_0$ and $D_{r+1}$ are the Zariski closure of the $1$-dimensional orbits of $X_0\simeq \cp^2$. 

Let $\Sigma$ be the fan of $X$. We denote by $\rho_i\in\Sigma$ the ray associated to the toric divisor $D_i$. Since $\Sigma$ is regular, we can order the $D_i$'s in such a way that the generators $\eta_i$ of the monoids $\rho_i\cap \zp^2$ satisfy 
$$
\det(\eta_{i},\eta_{i+1})=1
$$
(with convention $\eta_{r+2}=\eta_{0}$). We denote by $U_i$
the affine toric chart associated to the two-dimensional cone $\rho_{i}\rp^+\oplus \rho_{i+1}\rp^+$. Thus, 
$$
U_i= \Spec\cp[x_i,y_i]\simeq \cp^2,
$$
where torus coordinates $t=(t_1,t_2)$ and affine coordinates $(x_i,y_i)$ are related by relations
$$
t^m=x_i^{\langle m,\eta_i \rangle} y_i^{\langle m,\eta_{i+1}\rangle}
$$
for all $m=(m_1,m_2)\in\zp^2$, where $t^m=t_1^{m_1}t_2^{m_2}$ (see \cite{Danilov:gnus}). With these conventions, the toric divisors $D_{i}$ and $D_{i+1}$ have respective affine equations 
$$
D_{i|U_i}=\{x_i=0\}\quad{\rm and}\quad D_{i+1|U_i}=\{y_i=0\}
$$ 
and $|D_j|\cap U_i=\emptyset$ for $j\ne i,i+1$. 

We can associate to any toric divisor $E=\sum_{i=0}^{r+1} e_i D_i$ a polytope 
$$
P_E=\{m\in\rp^2,\langle m,\eta_i\rangle + e_i \ge 0, \,\,i=0,\ldots,r+1 \},
$$
where $\langle\cdot,\cdot \rangle $ designs the usual scalar product in $\rp^2$. Each Laurent monomial $t^m$ defines a rational function on $X$ with divisor
$$
\Div(t^m)=\sum_{i=0}^{r+1} \langle m,\eta_i \rangle D_i
$$
and there is a natural isomorphism
\begin{eqnarray}
H^0(X,\mathcal{O}_X(E))\simeq \bigoplus_{m\in P_{E}\cap \zp^2} \cp \ \cdot\,t^m,
\end{eqnarray}
where $\mathcal{O}_X(E)$ is the sheaf of meromorphic functions with polar locus bounded by $E$. 
Similary, the rational form  $\Psi_m$ defined by
$$
\Psi_{m|(\cp^*)^2}= t^m\frac{dt_1\land dt_2}{t_1 t_2}
$$
has divisor
$
\Div(\Psi_m)=\Div(t^m)-K_X
$
where 
$
-K_X:=D_0+\cdots +D_{r+1}
$
is an anticanonical divisor, and there is equality
\begin{eqnarray}
H^0(X,\Omega_X^2(E))=\bigoplus_{m\in P_{E+K_X}\cap \zp^2} \cp \ \cdot\,\Psi_m.
\end{eqnarray}
\vskip0mm
\noindent
\subsubsection{Explicit osculating criterions}
Let $(D,\gamma)$ be an osculation data on $\partial X$. We note $\Gamma=\gamma\cdot \partial X$ and $\Gamma_i=\gamma\cdot D_i$. By $(8)$, the conditions $(2)$ of Theorem $1$ are equivalent to that 
\begin{eqnarray}
\sum_{p\in|\Gamma|} \langle \gamma,\Psi_m \rangle_p =0 \quad \forall m\in P_{D+K_X}\cap \zp^2.
\end{eqnarray}
We give here an explicit formula for the residues $\langle \gamma,\Psi_m \rangle_p$ when the following hypothesis holds:
\vskip2mm
\noindent
\begin{center}
$(H_1)\quad$ {\it The zero-cycle $\Gamma$ is reduced and does not contain any torus fixed points.}
\end{center}
\vskip2mm
\noindent
In such a case, the germs $\wt{\gamma}_p$ are smooth, irreducible, and have transversal intersection with $\partial X$. For each $p\in|\Gamma|$, there is an unique component $D_i$ containing $p$ and we can choose the Weiertrass equation
$$
f_p(x_i,y_i)=y_i-\phi_p(x_i)
$$
for the germ $\wt{\gamma}_p$ in the affine coordinates of $U_i$.  The analytic function $\phi_p\in\cp\{x_i\}$ is well-defined modulo $(x_i^{k_i+1})$ and does not vanish at $0$. We obtain the following

\vskip4mm

\begin{prop} 
Suppose that $\Gamma$ satisfies $(H_1)$ and let $m\in P_{D+K_X}\cap \zp^2$. Then,
\begin{eqnarray}
\langle \gamma,\Psi_m \rangle_p =\frac{1}{(-\langle m,\eta_{i}\rangle) !}\frac{\partial^{-\langle m,\eta_{i}\rangle}}{\partial x_i^{-\langle m,\eta_{i}\rangle}}\bigg(\frac{\phi_{p}^{\langle m,\eta_{i+1}\rangle}}{\langle m,\eta_{i+1}\rangle}\bigg)(0)
\end{eqnarray}
for all $p\in|\Gamma_i|$ and all $i=1,\ldots,r$. We use here conventions $\phi_{p}^{k}/k:=\log(\phi_p)$ for $k=0$ and $\partial^{a}/\partial x^a(\phi_{p}^{k}/k):=0$ for $a<0$.
\end{prop}

\vskip2mm
Since $0\notin P_{D+K_X}$, equality $\langle m,\eta_{i+1}\rangle=0$ implies $\langle m,\eta_{i}\rangle\ne 0$ and previous expression is well-defined since it depends only on the derivatives and logarithmic derivatives of $\phi_p$ evaluated at $0$ (recall that $\phi_p(0)\ne 0$). Moreover, $-\langle m,\eta_{i}\rangle\le k_i$ and $(10)$ only depends on $\phi_p$ modulo $(x_i^{k_i+1})$

\vskip4mm
\noindent
{\it Example 2 (Wood's theorem).} We keep the same hypothesis and notations of Example $1$, with now an osculating order $k\ge 2$. Thus $D=(k+1)L$. Hypothesis $(H_1)$ holds and by  Proposition $1$ , there exists an osculating divisor for $(D,\gamma)$ if and only if 
\begin{eqnarray}
\frac{\partial^{m_1+m_2}}{\partial x^{m_1+m_2}}\Big(\sum_{p\in|\Gamma|}\frac{\phi_{p}^{m_1}}{m_1}\Big)(0)=0,\quad \forall\, m\in(\np^*)^2,\,m_1+m_2\le k.
\end{eqnarray}
Since
$
H^1(\mathcal{O}_{\pp^2}(d-k-1))=0
$
it follows that $(11)$ is also sufficient for the existence of an osculating \textit{curve}. When $k\le d+1$, we recover a theorem of Wood, \cite{Wood2:gnus}. 
\vskip2mm
\noindent

\subsubsection{Proof of Proposition $1$ } 
Let $m\in P_{D+K_X}$. In particular $\langle m,\eta_0 \rangle \ge 1$ and we can suppose that $m_2\ne 0$. The form $\Psi_m$ is holomorphic in $X_0$ and $\Psi_{m|X_0}=d(\psi_{m|X_0})$ where $\psi_m$ is the rational $1$-form on $X$ determined by its restriction
$$
\psi_{m|(\cp^*)^2}=\frac{t^m }{m_2}\frac{dt_1}{t_1}
$$
to the torus. Let $p\in |\Gamma_i|$. We compute the associated residue in the chart $U_i$. If we let $(e_1,e_2)$ be the canonical basis of $\rp^2$, we obtain 
\begin{eqnarray*}
\psi_{m|U_i} = x_i^{\langle m,\eta_i \rangle} y_i^{\langle m,\eta_{i+1} \rangle}\frac{1}{m_2}\big[\langle e_1,\eta_{i}\rangle \frac{dx_i}{x_i}+\langle e_1,\eta_{i+1}\rangle \frac{dy_i}{y_i}\big].
\end{eqnarray*}
If $\langle m,\eta_i \rangle > 0$, the form $\psi_m$ is holomorphic at $p \in|\Gamma_i|$ and $\res_p(\frac{df_p}{f_p}\land\psi_m)=0$. Suppose now that $\langle m,\eta_i \rangle \le 0$. The Cauchy integral representation for Grothendieck residues of meromorphic forms does not depend on $\underline{\ep}$ for $\ep_i$ small enough, thanks to Stokes Theorem. Thus, for sufficiently small $\ep_i$'s, we obtain
\begin{eqnarray*}
&& \res_p\Big[x_i^{\langle m,\eta_i \rangle} y_i^{\langle m,\eta_{i+1} \rangle}\frac{df_p}{f_p}\land \frac{dy_i}{y_i}\Big]\\
&=& -\frac{1}{(2i\pi)^2} \int_{|x_i|=\ep_1,|y_i-\phi_p|=\ep_2} x_i^{\langle m,\eta_i \rangle} y_i^{\langle m,\eta_{i+1} \rangle}\frac{d\phi_p}{y_i-\phi_p}\land \frac{dy_i}{y_i}\\
&=& \frac{1}{(2i\pi)^2} \int_{|x_i|=\ep_1} \Bigg(\int_{|y_i-\phi_p|=\ep_2} y_i^{\langle m,\eta_{i+1} \rangle-1}\frac{dy_i}{y_i-\phi_p}\Bigg)x_i^{\langle m,\eta_i \rangle}d\phi_p\\
&=&  \frac{1}{2i\pi} \int_{|x_i|=\ep} x_i^{\langle m,\eta_i \rangle}\phi_p^{\langle m,\eta_{i+1} \rangle-1}\phi_p'dx\\
&=&  \frac{1}{(-\langle m,\eta_{i}\rangle-1) !}\Bigg[\frac{\partial^{-\langle m,\eta_{i}\rangle}}{\partial x_i^{-\langle m,\eta_{i}\rangle}}\Bigg(\frac{\phi_{p}^{\langle m,\eta_{i+1}\rangle}}{\langle m,\eta_{i+1}\rangle}\Bigg)\Bigg]_{x_i=0}
\end{eqnarray*}
where the two last equalities are application of the Cauchy formula. In the same way, but simpler, we find
\begin{eqnarray*}
\res_p\Big[x_i^{\langle m,\eta_i \rangle} y_i^{\langle m,\eta_{i+1} \rangle}\frac{df_p}{f_p}\land \frac{dx_i}{x_i}\Big]=
\frac{-1}{(-\langle m,\eta_{i}\rangle)!}\Bigg[\frac{\partial^{-\langle m,\eta_{i}\rangle}}{\partial x_i^{-\langle m,\eta_{i}\rangle}}\big(\phi_{p}^{\langle m,\eta_{i+1}\rangle}\big)\Bigg]_{x_i=0},
\end{eqnarray*}
the minus sign coming from the chosen ordering of the numerator's factors. We deduce
$$
\res_p\Big(\frac{df_p}{f_p}\land\psi_m\Big)= \frac{C}{m_2}\times \frac{1}{(-\langle m,\eta_{i}\rangle) !}\Bigg[\frac{\partial^{-\langle m,\eta_{i}\rangle}}{\partial x_i^{-\langle m,\eta_{i}\rangle}}\Bigg(\frac{\phi_{p}^{\langle m,\eta_{i+1}\rangle}}{\langle m,\eta_{i+1}\rangle}\Bigg)\Bigg]_{x_i=0}
$$
where
$$
C=\langle e_1,\eta_{i}\rangle \langle m,\eta_{i+1}\rangle-\langle e_1,\eta_{i+1}\rangle\langle m,\eta_{i}\rangle=m_2 \det(\eta_i,\eta_{i+1})=m_2.
$$
This ends the proof of Proposition $1$. $\hfill{\square}$
\vskip4mm
\noindent

If two algebraic curves of fixed degree osculate each other on a finite subset with sufficiently big contact orders, they necessarily have a common component. We show in the next section that this basic fact permits to apply Theorem $1$ and Proposition $1$ to the Computer Algebra problem of polynomial factorization.

\section{Application to polynomial factorization.} 

This section is devoted to develop a new algorithm to compute the absolute factorization of a bivariate polynomial $f$. The method we introduce is comparable to a toric version of the Hensel lifting. It is in the spirit of the Abu Salem-Gao-Lauder algorithm \cite{Gao:gnus}, but uses vanishing-sums as in the Galligo-Rupprecht algorithm \cite{GR:gnus}. Our main contribution is that the underlying algorithmic complexity depends now on the Newton polytope $N_f$ instead of the degree $deg(f)$. We prove our main results Theorem $2$ and Proposition $2$ in Subsections $3.2$ and $3.3$. We describe the algorithm in Subsection $3.4$ and we comment and compare it with related results in Subsection $3.5$. We discuss non-toric singularities in Subsection $3.6$ and we conclude in the last Subsection $3.7$.

\subsection{Preliminaries and notations} Let $f\in K[t_1,t_2]$ be a bivariate polynomial with coefficients in some subfield $K\subset \cp$. By absolute factorization of $f$, we mean its irreducible decomposition in the ring $\cp[t_1,t_2]$, computed by floatting calculous with a given precision. 

Suppose that $f$ has the monomial expansion
$$
f(t)=\sum_{m\in \np^2} c_m t^m.
$$
We denote by $N_f$ the Newton polytope of $f$, convex hull of the exponents $m\in\np^2$ for which $c_m\ne 0$. We recall that by a theorem of Ostrowski \cite{O:gnus}, we have relations
$$
N_{f_1f_2}=N_{f_1}+N_{f_2}
$$
for any polynomials $f_1, f_2$, where $+$ designs here the Minkowski sum. 
\vskip2mm
\noindent

We assume that the following hypothesis holds.
\vskip2mm
\begin{center}
 $(H_2)\quad$\textit{The Newton polytope of $f$ contains the origin.}
\end{center}
\vskip2mm
This equivalent to that $f(0,0)\ne 0$. 

\begin{rema} We can always find $s\in \zp^2$ so that the Laurent polynomial $ft^{-s}$ becomes a polynomial $f'\in \cp[u_1,u_2]$ with $f'(0,0)\ne 0$ after some \textit{monomial} change of affine coordinates $(t_1,t_2)\mapsto (u_1,u_2)$. This transformation preserves the sparse structure of $f$ so that hypothesis $(H_2)$ is not restrictive for our purpose. 
\end{rema}
An \textit{exterior facet} of $N_f$ is a one-dimensional face $F$ of $N_f$
whose normal inward primitive vector has at least a negative coordinate. The associated normal ray of $N_f$ is called an \textit{exterior ray}.
The associated \textit{facet polynomial} $f_F$ of $f$ is defined to be
$$
f_F=\sum_{m\in F\cap \zp^2} c_m t^m.
$$
The facet polynomials have a one dimensional Newton polytope and become univariate polynomials after a monomial change of coordinates. 
\vskip2mm
\noindent

We construct now a toric completion of $\cp^2$ associated to the exterior facets of $N_f$. To this aim, we  consider the \textit{complete} simplicial fan $\Sigma_f$ of $\rp^2$ whose rays are
$$
\Sigma_f(1)=\{exterior\,\, rays\,\,of\,\, N_f\}\,\cup\,\{(0,1)\rp^+,(1,0)\rp^+\}.
$$
The toric surface 
$
X_f=X_{\Sigma_f}
$ 
is a simplicial toric completion of $\cp^2= Spec\, \cp [t_1,t_2]$ whose boundary $\partial X_f=X_f\setminus \cp^2$ is the sum of the irreducible toric divisors associated to the exterior rays of $N_f$. The intersection of the Zariski closure $C_f\subset X_f$ of the affine curve $\{f=0\}\subset\cp^2$ with the irreducible components of the boundary $\partial X_f$ corresponds to the \textit{non trivial} roots of the exterior facet polynomial of $f$ (see \cite{Kho:gnus} for instance). In particular, the curve $C_f$ does not contain any torus fixed points (it's clear for those contained in $\partial X_f$ and the remaining torus fixed point of $X_0$ does not belong to $C_f$ by $(H_2)$). 

We say  that $f$ satisfies hypothesis $(H_1)$ of Subsection $2.5$ when the zero-cycle $\Gamma_f=C_f\cdot \partial X_f$ does. This corresponds to the case of exterior facet polynomials of $f$ with a square free absolute decomposition in $\cp[t,t^{-1}]$.


The boundary of $X_{f}$ contains the singular locus $S:=Sing(X_f)$ of $X_f$ and is generally not a normal crossing divisor. To this aim, we consider a toric desingularization 
$$
X\stackrel{\pi}{\rightarrow} X_f
$$
of $X_f$, whose exceptional divisor $E$ satisfies $S=\pi(|E|)$. The smooth surface $X$ is now a projective toric $\cp^2$-completion with a toric normal crossing boundary
$$
\partial X=D_1+\cdots +D_r
$$
whose support contains $|E|$.  Since $C_f\cap S=\emptyset$, the total transform $C=\pi^{-1}(C_f)$ of $C_f$ has a proper intersection with  $\partial X$ and the irreducible decomposition of $C$ corresponds to the absolute factorization of $f$. Note that  $C$ is not necessarily reduced. 

As in Subsection $2.5$, we denote by $\rho_0,\ldots,\rho_{r+1}$ the rays of the fan $\Sigma$ of $X$, and by $\eta_0,\ldots,\eta_{r+1}$ their primitive vectors. So $\eta_0=(0,1)$ and $\eta_{r+1}=(1,0)$. The curve $C$ has a positive intersection $C\cdot D_i >0$ when $\rho_i\in\Sigma(1)$ is an exterior ray of $N_f$, and $C\cdot D_i=0$ 
when $\rho_i\in \Sigma(1)\setminus\Sigma_f(1)$ (that is when $\pi(|D_i|)\subset S$). 

The following lemma will be usefull:
\begin{lemm}
There is an unique effective divisor $D$ linearly equivalent to $C+\partial X$ with support $|\partial X|$ and 
$$
P_{D+K_X}\cap\zp^2=N_f\cap (\np^*)^2.
$$
\end {lemm}
\vskip1mm
\noindent

\begin{proof}
The polynomial $f$ extends to a rational function on $X$ whose polar divisor 
$\Div_{\infty}(f)$ 
is  supported by $|\partial X|$. Thus, the divisor 
$$
D:=\Div_{\infty}(f)+\partial X
$$
is effective, with support $|\partial X|$, and rationally equivalent to $C+\partial X$. If $D'$ is an other candidate, then
$D-D'=\Div(h)$ for some $h\in\cp(X)$ with no poles nor zeroes outside $|\partial X|$. Since $X\setminus |\partial X|\simeq \cp^2$, the rational function $h$ is necessarily constant and $D=D'$. 
By using affine charts, we check easily that the polar divisor of $f$ is equal to 
$$
\Div_{\infty}(f)=\sum_{i=1}^r k_i D_i,\quad k_i=-\min_{m\in N_f}\langle m,\eta_i\rangle.
$$
By $(7)$, it follows that $N_f$ is contained in the polytope $P_G$ of $G=\Div_{\infty}(f)$ and intersects each face of $P_G$. Since the fan of $X$ refines noth normal fans of  $G$ and $N_f$ (we use here hypothesis $(H_2)$), it follows that
$
P_G=N_f.
$
Then, equality $D+K_X=G-D_0-D_{r+1}$ implies that
$$
P_{D+K_X}=\{m\in P_G,\,\langle m,\eta_0\rangle \ge 1,\,\,\langle m,\eta_{r+1}\rangle \ge 1\}=P_G\cap (\np^*)^2.
$$
\end{proof}

\vskip2mm
\noindent

\subsection{Computing the Newton polytopes of the absolute factors}
We show here how to recover the Newton polytopes of the irreducible absolute factors of $N_f$. 

We recall that the Newton polytope of a factor of $f$ is necessarily a Minkowski summand of $N_f$.
For a general polytope $Q$, we denote by $Q^{(i)}$ the set of $m\in Q$ for which the scalar product $\langle  m,\eta_i\rangle$ is minimal. 

We obtain the following 

\vskip6mm

\begin{theo} 
1. Let $f\in K[t_1,t_2]$ which satisfies $(H_2)$. Let $\pi$, $X$ and $C$ be as before and denote by 
$\gamma$ the restriction of $C$ to the divisor $D$ of Lemma $3$. Let $Q$ be a lattice Minkowski summand of $N_f$. 

There exists an absolute factor $q$ of $f$ with Newton polytope $Q$ if and only if there exists $\gamma'\le \gamma$ an effective Cartier divisor on $D$ with 
\begin{eqnarray}
\deg(\gamma' \cdot D_i)=\Card(Q^{(i)}\cap \zp^2)-1,\,\,i=1,\ldots,r
\end{eqnarray}
and such that $(2)$ holds for the osculation data $(D,\gamma')$ on the boundary of $X$. The polynomial $q$ can be recovered from $\gamma'$.
\end{theo}

\vskip4mm

Note that $Q^{(i)}$ is reduced to a point and $Card(Q^{(i)}\cap \zp^2)-1=0$ as soon as $\rho_i$ is not an exterior ray of $N_f$.

\vskip2mm

\subsubsection*{Proof} A Minkowski sum decomposition $N_f=P+Q$ corresponds to a line bundle decompostion
$$
\mathcal{O}_X(G)=\mathcal{O}_X(G_P)\otimes \mathcal{O}_X(G_Q),
$$
where $G$, $G_P$ and $G_Q$ are the polar divisors of the rational functions determined by polynomials with respective polytope $N_f$, $P$ and $Q$. Note that $G=G_P+G_Q$. Since $\Sigma$ refines the normal fan of each polytope, we have equalities $N_f=P_G$, $Q=P_{G_Q}$ and $P=P_{G_P}$ (see Lemma $3$) and all involved line bundles are globally generated over $X$ (see \cite{F:gnus}). 

A factor of $f$ with Newton polytope $Q$ corresponds to a divisor $C'\le C$ lying in the complete linear system $|\mathcal{O}_X(G_Q)|$. It defines a Cartier divisor $\gamma'=C_{|D}'\le C_{|D} = \gamma$ on $D$. Thus, conditions $(2)$ hold for the osculation data $(D,\gamma')$ and
\begin{eqnarray*}
\deg(\gamma'\cdot D_i)= \deg(C'\cdot D_i)  =  \deg(\mathcal{O}_X(G_Q)_{|D_i})
\end{eqnarray*}
for all $i=1,\ldots,r$. Since $\mathcal{O}_X(G_Q)$ is globally generated over $X$, toric intersection theory gives equalities
$$
\deg(\mathcal{O}_X(G_Q)_{|D_i}) = \Card(Q^{(i)}\cap \zp^2)-1
$$
for all $i=1,\ldots,r$. This shows necessity of conditions. 

Suppose now that $(D,\gamma')$ satisfies $(12)$ and $(2)$ and let $F$ be an osculating divisor asociated to that data. We deduce equalities
$$
\deg(\mathcal{O}_X(F)_{|D_i})=\deg(\gamma'\cdot D_i)=\deg(\mathcal{O}_X(G_Q)_{|D_i})
$$
for all $i=1,\ldots,r$. Thus, there is an isomorphism
$$
\mathcal{O}_{\partial X}(F)\simeq \mathcal{O}_{\partial X}(G_Q)
$$
and $\mathcal{O}_{X}(F)\simeq \mathcal{O}_X(G_Q)$ by Corollary $1$ of Subsection $2.4$. Finally, $(2)$ and $(12)$ hold for $(D,\gamma')$ if and only if 
$
\mathcal{O}_{D}(\gamma')\simeq \mathcal{O}_D(G_Q).
$
Let us show that there exists in such a case an \textit{effective} osculating divisor. The problem is that $H^1(X,\mathcal{O}_X(G_Q-D))$ does not necessarily vanish. To avoid this difficulty, we introduce the effective divisor 
$$
D_Q:=G_Q+\partial X \le G+\partial X = D.
$$
Then $\mathcal{O}_{D_Q}(G_Q)\simeq \mathcal{O}_{D_Q}(\gamma'_{|D_Q})$ and
$
H^1(X,\mathcal{O}_X(G_Q-D_Q))\simeq H^1(X,\mathcal{O}_X(-\partial X)).
$
Since $\partial X$ is reduced and connected, the restriction map 
$H^0(\mathcal{O}_X)\rightarrow H^0(\mathcal{O}_{\partial X})$ is an isomorphism and we deduce equality
$
H^1(X,\mathcal{O}_X(-\partial X))=0.
$
By Theorem $1$, there thus exists $C'\in |\mathcal{O}_X(G_Q)|$ with 
$$
C'_{|D_Q}=\gamma'_{|D_Q}.
$$
Such a curve has  proper intersection with the remaining toric divisors $D_0$ and $D_{r+1}$ (otherwise the support of $\gamma'$ would contain a torus fixed point). We deduce that there exists a unique Cartier divisor $\gamma_0$ on $D+D_0+D_{r+1}$ so that
$$
\gamma_{0|D}=\gamma'\quad{\rm and}\quad \gamma_{0|D_0+D_{r+1}}=C'\cdot D_0+C'\cdot D_{r+1}
$$
We have equality $D+D_0+D_{r+1}=G-K_X$ and we obtain
\begin{eqnarray*}
H^1(X,\mathcal{O}_X(G_Q-D-D_0-D_{r+1})& \simeq & H^1(X,\mathcal{O}_X(G_Q-G+K_X)\\
&\simeq & H^1(X,\mathcal{O}_X(G-G_Q))^{\check{}}\\
&\simeq & H^1(X,\mathcal{O}_X(G_P))^{\check{}}=0.
\end{eqnarray*}
The second isomorphism is Serre Duality, and the last equality comes from the fact that the $H^1$ of a globally generated line bundle on a toric surface vanishes (see \cite{F:gnus}). Last equality combined with the short exact sequence
$$
0\rightarrow \mathcal{O}_X(G_Q-D-D_0-D_{r+1})\rightarrow \mathcal{O}_X(G_Q)\rightarrow \mathcal{O}_{D+D_0+D_{r+1}}(\gamma_0)\rightarrow 0,
$$
implies that there exists $C''\in |\mathcal{O}_X(G_Q)|$ which restricts to $\gamma_0$ on $D+D_0+D_{r+1}$. A fortiori, $C''$ restricts to $\gamma'$ on $D$. Since $G_Q-D_Q= -\partial X< 0$, we have $H^0(X, \mathcal{O}_X(G_Q-D_Q))=0$. Thus $C'=C''$, the two curves having the same restriction to $D_Q$. 

There remains to show that $C'\le C$. Suppose that there exists an irreducible component $C_0$ of $C'$ with proper intersection with $C$. Then, we can compute the intersection multiplicity of $C_0$ and $C$ at any $p\in X$. It is given by
$$
\mult_p(C_0,C)=\dim_{\cp}\frac{\mathcal{O}_{X,p}}{(f_p,g_p)}
$$
where $f_p$ and $g_p$ are respective local equations for $C$ and $C_0$ at $p$. By assumption, 
$$
C_{0|D}\le C_{|D}
$$
so that the class of $g_p$ in $\mathcal{O}_{D,p}$ divides that of $f_p$. We deduce that
$$
f_p=u_p g_p+ v_p h_p
$$
for some holomorphic functions $u_p$, $v_p$, where $h_p$ is a local equation for $D$. Thus, we obtain inequality
$$
\dim_{\cp}\frac{\mathcal{O}_{X,p}}{(f_p,g_p)}\ge \dim_{\cp}\frac{\mathcal{O}_{X,p}}{(g_p,h_p)} = \mult_{p}(C_0,D),
$$
for all $p$ in the support of $D$. Finally,
$$
\deg(\mathcal{O}_X(C_0))_{|C} \ge \sum_{p\in|D|} \mult_p(C_0,C)\ge \sum_{p\in|D|} \mult_p(C_0,D)= \deg(\mathcal{O}_X(C_0))_{|D}.
$$
Since $D$ is rationally equivalent to $C+\partial X$, this implies that $C_0$ has a negative intersection with the boundary. This leads to a contradiction. Thus $C_0$ is necessarily an irreducible component of $C$. Finally we have $C'\le C$, corresponding to an absolute factor $q$ of $f$ with Newton polytope $Q$. 

Since $G_Q-D_Q<0$, the restriction $H^0(\mathcal{O}_X(G_Q))\rightarrow H^0(\mathcal{O}_{D_Q}(G_Q))$ is injective. Thus $C'\in|\mathcal{O}_X(G_Q)|$ can be uniquely recovered from its restriction $\gamma'_{|D_Q}$ to $D_Q$ (and a fortiori from $\gamma'$). $\hfill{\square}$

\vskip2mm
\noindent

\begin{rema}
If we only assume that $N_f$ intersects both coordinate axes, the osculating criterions still permit to detect an absolute factor $q$ of $f$ associated to the choice of a summand $Q$ of $N_f$. Nevertheless, although the Newton polytope $N_q$ of $q$ still has the same exterior facets of $Q$, we can not conclude that $N_q=Q$. 
\end{rema}

\vskip0mm
\noindent

\subsection{Computing the absolute factors of $f$.}

Suppose that $\gamma'\le \gamma$ satisfies conditions of Theorem $2$ for a  Minkowski summand $Q$ of $N_f$. We give here an efficient way to compute the corresponding absolute factor $q$ of $f$. 
The polynomial $q$ admits the monomial $\cp$-expansion
$$
q(t)=\sum_{m\in Q\cap\np^2} a_m t^m,
$$
and we look for the homogeneous class 
$
[a]\in \pp^{\Card(Q\cap\np^2)-1}(\cp)
$
of $a=(a_m)_{m\in Q\cap\np^2}$.
\vskip2mm
\noindent

We denote by $\Gamma'=\gamma'\cdot \partial X$ and by $\Gamma'_i=\Gamma'\cdot D_i$. Let us fix $p\in |\Gamma'_i|$. 
If $f$ satisfies the hypothesis $(H_1)$, the germ of $C$ at $p$ has the Weirstrass equation 
$$
y_i-\phi_p(x_i)=0,
$$
and we can define
$$
\alpha_p(u,v):=\frac{\partial^{u}\phi_{p}^{v}}{\partial x_i^{u}}(0)\quad{\rm if}\,\,u\ge 0, \qquad 
\alpha_p(u,v):=0\quad {\rm if}\,\,u<0,
$$
for all integers $u$, $v$ (recall that $\phi_p(0)\ne 0$). For $k\in \np$ and $m\in \zp^2$, we then define the complex number
$$
\beta_p(k,m):=\alpha_p(k-\langle m,\eta_i\rangle-e_i,\langle m,\eta_{i+1}\rangle+e_{i+1}),
$$
\vskip1mm
\noindent
where $e_i:=- \min_{m\in Q} \langle m,\eta_i\rangle$. We denote by $\Vol(Q)$ the euclidean volume of $Q$ in $\rp^2$. 
We obtain the following

\vskip3mm
\noindent
\begin{prop}
Suppose that $f$ satisfies $(H_1)$ and $(H_2)$. The homogeneous vector $[a]$ is uniquely determined by conditions
\begin{eqnarray*}
\sum_{m\in Q \cap \zp^2} a_m \beta_p(k,m)=0,
\end{eqnarray*}
for all $p\in|\Gamma'_i|$, all $0\le k\le  e_i$ and all $i=1,\ldots,r$. 
The underlying linear system $(S_{\gamma'})$ contains $2\Vol(Q)+\deg(\Gamma')$ equations of $\Card(Q\cap \zp^2)$ unknowns and only depends on the restriction of $\gamma'$ to the divisor $G_Q+\partial X$. 
\end{prop}
\vskip3mm
\noindent

\begin{proof}
By \cite{Danilov:gnus}, the curve $C'$ of $q$ has affine polynomial equation $q_i=0$ in the chart $U_i=\Spec\cp[x_i,y_i]$, where
\begin{eqnarray*}
q_i(x_i,y_i)=\sum_{m\in Q \cap \zp^2} a_m x_i^{\langle m,\eta_i\rangle+e_i}y_i^{\langle m,\eta_{i+1}\rangle+e_{i+1}}.
\end{eqnarray*}
Let $p\in|\Gamma'_i|$. Since the germ of $C'$ at $p$ is contained in that of $C$ it follows that
$$
q_i(x_i,\phi_p(x_i))\equiv 0
$$ 
It's easy to show that the $k^{th}$-derivative of this expression evaluated at $0$ is equal to $c\times\sum_{m\in Q \cap \zp^2} a_m \beta_p(k,m)$ for a non zero scalar $c$. Thus the coefficients of $[a]$ determine a non trivial solution of $(S_{\gamma'})$. 

If $\wt{a}=(\wt{a}_m)_{m\in Q\cap \zp^2}$ is an other non trivial solution of $(S_{\gamma'})$, the polynomial 
$
\wt{q}(t)=\sum_{m\in Q\cap \zp^2} \wt{a}_m t^m
$
satisfies
$$
\Big[\frac{\partial^{k}}{\partial x_i^{k}}  \wt{q}_i(x_i,\phi_p(x_i))\Big]_{x_i=0}=0
$$ 
for all $p\in|\Gamma'_i|$, all $0\le k\le e_i$ and all $i=1,\ldots,r$. 
By the duality theorem, this forces the germ of $\wt{q}_{i}$ at $p$ to belong to the ideal $(y_i-\phi_p(x_i),x_i^{e_i+1})$. It follows that the curve $\wt{C}\subset X$ of $\wt{q}$ satisfies
$$
\wt{C}_{|\sum (e_i+1)D_i}\ge C'_{|\sum (e_i+1)D_i}.
$$
\vskip1mm
\noindent
Since $N_{\wt{q}}\subset Q=N_q$, it follows that $deg(\wt{C}\cdot D_i)\le deg(C'\cdot D_i)$ for all $i$. Combined with previous inequality, this forces equality 
$$
C'_{|\sum (e_i+1)D_i}= \wt{C}_{|\sum (e_i+1)D_i}.
$$
We have chosen the $e_i$'s in order to that 
$$
\sum_{i=1}^{r} (e_i+1)D_i =G_Q+\partial X,
$$
and both curves $C'$ and $\wt{C}$ belong to the linear system $|\mathcal{O}_X(G_Q)|$. The restriction map $H^0(X,\mathcal{O}_X(G_Q))\rightarrow H^0(X,\mathcal{O}_{G_Q+\partial X}(G_Q))$ being injective, the previous equality forces $C'=\wt{C}$. It follows that $\wt{q}=cq$ for some $c\in\cp^*$ and $[a]=[\wt{a}]$.
\vskip1mm
\noindent

By construction, the linear system $(S_{\gamma'})$ only depends on the restriction of $\gamma'$ to the divisor $G_Q+\partial X< G+\partial X=D$. It contains precisely
$$
\sum_{i=1}^r \sum_{p\in |\Gamma'_i|}(e_i+1)=\sum_{i=1}^r (e_i+1)\deg(C'\cdot D_i)=\deg(C'\cdot G_Q)+ \deg(\Gamma')
$$
equations and $\deg(C'\cdot G_Q)=\deg(C'\cdot C') = 2 \Vol(Q)$ (see \cite{F:gnus} for instance).
\end{proof}

\vskip2mm
\noindent
\begin{rema}
Each linear equation in Proposition $2$ involves a reduced number of unknowns and the linear system $(S_{\gamma'})$ has a very particular sparse structure. For instance, letting $k=0$, we obtain for each $i=1,\ldots,r$ the linear subsystem
$$
\sum_{m\in Q^{(i)}\cap \zp^2}a_m [\phi_p(0)]^{\langle m,\eta_{i+1}\rangle+e_{i+1}}=0, \,\,\forall\,\,p\in |\Gamma'_i|.
$$
It has $Card (Q^{(i)}\cap \zp^2)-1$ equations with $Card(Q^{(i)}\cap\zp^2)$ unknowns and permits to determine the coefficients of the $i^{th}$ exterior facet polynomial of $q$. For $k=1$, we deduce relations on the coefficients $\{a_m,\,\,\langle m,\eta_i\rangle+e_i =1\}$. For a general $k$, 
we deduce relations on the coefficients $\{a_m,\,\,\langle m,\eta_i\rangle+e_i =k\}$. 
In general, there are no non trivial solutions to $(S_{\gamma'})$. The vector subspace of solutions has dimension $1$ precisely when $\gamma'$ obeys to the osculation conditions $(2)$ .

\end{rema}
\vskip1mm
\noindent
\subsection{A sparse vanishing-sums algorithm}

We describe here a sparse vanishing-sums algorithm associated to Theorem $2$ and Proposition $2$. When we say compute, we mean compute by using floatting calculous with a given precision. When we say test a vanishing-sum, we mean test an $\le \ep$-sum for an arbitrary small $\ep>0$. 
\vskip2mm
\noindent
\textit{Input:} A polynomial $f\in K[t_1,t_2]$ which satisfies $f(0,0)\ne 0$ and with square free exterior facet polynomials. 
\vskip1mm
\noindent
\textit{Output} : The irreducible factorization of $f$ over $\cp$.
\vskip2mm
\noindent
{\it Step 1. Compute the Minkowski summands of $N_f$.} Use for instance the algorithm presented in \cite{Gao2:gnus}. If $N_f$ is irreducible, so is $f$. Otherwise go to step $2$.
\vskip2mm
\noindent
{\it Step 2. Compute the fan of $X$.} The fan $\Sigma$ is obtained from $\Sigma_f$ by adding some rays in the singular two-dimensional cones of $\Sigma_f$. Such a fan can be obtained in a canonical way by using an Euclidean algorithm, or by computing some Hirzebruch-Jung continued fraction (see \cite{F:gnus}).
\vskip2mm
\noindent
{\it Step 3. Compute the osculating data $(D,\gamma)$.} By Lemma $3$, we have equality 
$$
D=\sum_{i=1}^r (k_i+1)D_i,\quad k_i:=-\min_{m\in N_f}\langle m,\eta_i\rangle.
$$
The Cartier divisor $\gamma$ on $D$ is obtained by computing the family of implicit functions 
$\{\phi_p, \,\,p\in|\Gamma_i|\}$ up to order $k_i$, for $i=1,\ldots,r$. Note that $\Gamma_i=0$ when $\rho_i$ is not an exterior ray of $N_f$.
\vskip2mm
\noindent
{\it Step 4. Compute the Newton polytopes of the absolute irreducible factors.} Proposition $1$ combined with Theorem $2$ gives an efficient way to compute the decompositions
$$
\gamma=\gamma_1+\cdots +\gamma_s,\qquad N_f=Q_1+\cdots +Q_s
$$ 
of $\gamma$ and $N_f$ associated to the irreducible absolute decomposition
$
f=q_1\cdots q_s
$ of $f$. 
\vskip2mm
\noindent
{\it Step 5. Compute the irreducible absolute factors of $f$.} Use Proposition $2$. The coefficients of the linear systems $(S_{\gamma_i})$, $i=1,\ldots,s$ are obtained from the residues $\langle \gamma,\Psi_m \rangle_p$, $m\in Q_i\cap (\np^*)^2$  already computed in step $4$.

\vskip4mm
\noindent

The numerical part of the algorithm reduces to the computation of the roots of the univariate exterior facet polynomials of $f$. Then it detects the absolute factorization of $f$ with a probability which increases after each positive vanishing-tests, up to obtain the adequat decomposition of $N_f$. Finally, we compute the factors by solving some linear systems. Roughly speaking, we obtain here a toric version of the Hensel lifting (see Remark $3$). A comparable algorithm has been obtained by Abu Salem-Gao-Lauderin in \cite{Gao:gnus}, by using combinatorial tools.

If we can decide formally if a sum vanishes, there is no chance of failure in step $4$ and the algorithm is deterministic. If we test the osculating criterions $(9)$ for a generic linear combination of the involved $m$'s, the Newton polytope decomposition is valid only with probability one, in the vain of the Galligo-Rupprecht and Elkadi-Galligo-Weimann algorithms \cite{GR:gnus} or \cite{EGW:gnus}. 

As in \cite{GR:gnus} and \cite{EGW:gnus}, our algorithm necessarily uses $\le \ep$-sum tests imposed by floatting calculous and numerical approximation. Thus, it can happen that the Newton decomposition of $N_f$ in step $4$ does not correspond to the absolute decomposition of $f$ (and that for any choice of $\ep>0$) and there is a chance of failure of the algorithm. Nevertheless, in the important case a polynomial $f$ defined and irreducible over $\qp$ the authors in \cite{CG:gnus} show that we can recover the exact factorization with formal coefficients in a finite extension of $K$ from a sufficiently fine approximate factorization. This problem will be explored in a further work.

\subsection{Comparison with related results}

We compare here our algorithm with that of Galligo-Rupprecht \cite{GR:gnus} and with that of Elkadi-Galligo-Weimann \cite{EGW:gnus}.

\subsubsection{Comparison with the Galligo-Rupprecht GR-algorithm} In \cite{GR:gnus}, the authors perform a generic change of affine coordinates and then compute the factorization of $f$
$$
f(t_1,t_2)=\prod_{i=1}^{d} (t_2-\phi_i(t_1))
$$
in $\cp\{t_1\}[t_2]$ modulo $(t_1^3)$. They detect the factors of $f$ with probability one by testing the Reiss relation $(3)$ on subfamilies $\mathcal{F}\subset \{\phi_i,\,i=1,\ldots,d\}$. Then they lift and compute the candidate factors by using Hensel lemma. Let us compare  with our approach.
\vskip2mm
\noindent
\textit{The recombination number.} We define the recombination number $\mathcal{N}(f)$  to be the maximal number of choices $\gamma'<\gamma$ necessary to detect the absolute factorization of $f$ in step $4$.
By $(12)$, it depends on the geometry of $N_f$ and is subject to constraints given by the possible Minkowski-sum decompositions of $N_f$. For instance, if $f$ is irreducible over $K$, its irreducible absolute factors necessarily have \textit{the same polytope} (see \cite{Ch:gnus} for instance) and $\mathcal{N}(f)$ decreases drastically. 

We denote by $\mathcal{M}(f)$ the recombination number of the GR-algorithm, that is the number of possible choices for the families $\mathcal{F}$ when taking in account restrictions imposed by the possible Minkowski-sums decompositions of $N_f$. 

We have the following improvement:
\vskip2mm
\noindent
\begin{lemm}  Suppose that $f$ is irreducible over $K$ and satisfies hypothesis $(H_1)$ and $(H_2)$. Let $d=\deg(f)$. Then
$$
\mathcal{N}(f)\le \mathcal{M}(f)\le 2^{d}
$$ 
with equivalence 
$$
\mathcal{N}(f)=\mathcal{M}(f)\iff N_f=\Conv\{(0,0),(d,0),(0,d)\},
$$
where $\Conv$ designs the convex hull.
\end{lemm}
\vskip2mm
\noindent

\begin{proof}
let $n$ be the biggest integer so that $N_f=nQ_0$ for a lattice polytope $Q_0$. Thus $d=nl$ where $l$ is the total degree of a polynomial with polytope $Q_0$. By \cite{Ch:gnus}, an irreducible absolute factor of $f$ has Newton polytope $kQ_0$ for some integer $k$ which divides $n$ and the GR-algorithm looks for an irreducible factor of degree $kl$. This gives the recombination number
$$
\mathcal{M}(f)=\sum_{k| n} C^{nl}_{kl},
$$
where $C^{i}_{j}$ is the usual number of combinations. We clearly have $\mathcal{M}(f)\le 2^{d}$. Equality $\mathcal{M}(f)=2^{d}$ holds if and only if $f$ is a dense polynomial which is not assumed to be irreducible over $K$. 

Denote now by $l_1,\ldots,l_t$ the lattice lenght of the exterior facets of $Q_0$. Then, restrictions $(12)$ imposed to the possible choice of $\gamma'\le \gamma$ induce equality
\begin{eqnarray*}
\mathcal{N}(f)=\sum_{k|n} \prod_{i=1}^t C^{nl_i}_{kl_i}.
\end{eqnarray*}
We have both inequalities 
\begin{eqnarray*}
\prod_{i=1}^t C^{nl_i}_{kl_i} \le  C^{n(l_1+\cdots+l_t)}_{k(l_1+\cdots+l_t)} 
\quad{\rm and}\quad C^{n(l_1+\cdots+l_t)}_{k(l_1+\cdots+l_t)} \le   C^{nl}_{kl}
\end{eqnarray*}
for any $k\le n$. First inequality is an equality if and only if there is only $t=1$ exterior facet, and the difference strictly increases with $t$. Second inequality follows from the fact that the sum 
$$
l_f=nl_1+\cdots+nl_t
$$
of the lattice lenght of the exterior facets of $N_f$ is smaller or equal to $d=\deg(f)$. Moreover, we can convince that there is equality $l_f=d$ if and only if the normal fan $\Sigma_f$ of $N_f$ is regular. This is of course exceptional, and $l_f<<d$ in general. Finally, we have $\mathcal{N}(f)\le \mathcal{M}(f)$ and equality holds if and only if $N_f$ is regular with one exterior facet.  Since $0\in N_f$ by assumption, this is equivalent to that 
$
N_f=\Conv\{(0,0),(d,0),(0,d)\}.
$
\end{proof}

Note that there are fast factorization algorithms over a number field $K$ (see \cite{BHKS:gnus}, \cite{Lecerf:gnus}). The following example illustrates Lemma $4$.
\vskip2mm
\noindent
{\it Example 3.}
Suppose that $f$ is irreducible over $K$ and that $N_f=nQ_0$, where $Q_0$ is the ``undivisible" lattice polytope  
$$
Q_0= \Conv\{(0,0),(a-1,0),(0,a-1),(a,a)\},
$$
with $a\ge 2$  and $n$ prime. The lattice lenghts of the $t=2$ exterior facets of $N_f$ are both equal to $n$. Since $n$ is prime, the underlying recombination number is equal to
$$
\mathcal{N}(f)= n^2.
$$
In particular, is does not depend on $a$. The GR-algorithm considers $f$ as a dense polynomial of degree $deg(f)=2na$ and looks for an irreducible absolute factor of degree $2a$. The induced recombination number $\mathcal{M}(f)$ is thus equal to
$$
\mathcal{M}(f)=C^{2an}_{2a}\simeq \frac{2a^{2a}}{(2a)!} n^{2a},
$$
and growths exponentially with $a$.

\vskip2mm
\noindent

The asymptotic estimation of $\mathcal{N}(f)$ for a ``generic polytope"\footnote{The genericity has to be defined relatively to some invariant, as the cardinality of interior lattice points, or the volume for instance.} $N_f$ is a difficult problem. It is related to the estimation of the number of exterior facets, and to the estimation of ``how singular"  is the normal fan. Geometrically, these numbers correspond to the Picard numbers of $X_f$ and of a minimal resolution $X$. 

Note that by Lemma $3$ and Theorem $2$, the maximal number of vanishing-sums to test is equal to 
$$
Card (N_f\cap (\np^*)^2)\times \mathcal{N}(f).
$$
If we only want to detect absolute factors with probability $1$, it's enough to test $(9)$ on a generic linear combination of the involved $m\in N_f\cap \zp^2$. In such a case, it's enough to test $\mathcal{N}(f)$ vanishing-sums. 
\vskip4mm
\noindent
\textit{The numerical part.} The numerical part of our algorithm reduces to the computation of the roots of the exterior facet univariate polynomials. It's faster to factorize $t$ univariate polynomials of degree $l_1,\ldots,l_t$ than an univariate polynomial of total degree $d\ge l_1+\cdots+l_t$. Thus, the computation of $C\cdot \partial X$ is faster than with a generic line as soon as $\mathcal{N}(f)<\mathcal{M}(f)$. 
\vskip4mm
\noindent
\textit{The lifting step.} 
Although we need to compute a reduced number implicit functions, we need \textit{in general} a bigger precision on the $\phi_p$'s than the  upper bound $\deg(f)$ precision required with the classical Hensel lifting. In Example $3$, we need for instance to compute the $\phi_p$'s up to order $na^2 >> 2na=\deg(f)$. Morally, the more the recombination number is reduced, the more the required precision on the $\phi_p$'s increases. Thus, we should be very carefull when comparing algorithmic complexity. Nevertheless, we gain on both sides in the important case of a polynomial of bidegree $(a,b)$. Formulas $(9)$ and $(10)$ shows that we need to compute $a$ and $b$ implicit functions up to respective maximal orders $b$ and $a$, while the GR-algorithm computes $a+b$ implicit functions up to a maximal order $a+b$.


\subsubsection{Comparison with the Elkadi-Galligo-Weimann EGW-algorithm}  
In \cite{EGW:gnus}, the authors develop the sketch of an algorithm with the same recombination number that here by using the  interpolation criterions obtained in \cite{W2:gnus}. Their approach necessites to compute numerically the intersection of $C\subset X$ with a \textit{generic} curve $L$ in a very ample linear system which is ``close enough" to $|\partial X|$. In other words, they pick a Newton polytope $P$ whose normal fan is that of $X$ and then solve a polynomial system $f=\ep p+1=0$, where $p$ has polytope $N_p=P$ and $\ep$ is a small positive real number. When $\ep$ goes to $0$, the roots of the system go to the boundary of $X$ and they deduce an asymptotic distribution of the zero-cycle $C\cdot L$ which traduces  the polytopal information. As in the GR-algorithm, the generic choice of $p$ in the EGW-algorithm permits to compute $\deg(C\cdot L)$ implicit functions only up to order $2$ in order to detect the absolute decomposition of $f$ with probability one. Roughly speaking, Theorem $2$ corresponds to the limit case $\ep=0$. A great advantage is that we avoid the delicate problem of the ``small enough $\ep$''  choice and of the asymptotic distribution lecture.  We compute only $\deg(C\cdot \partial X)<<\deg(C\cdot L)$ implicit functions but with a precision $>>2$ and we detect the absolute decomposition of $f$ deterministically.

\subsection{Using non toric information} Suppose that $C$ is reduced, but with singularities along the boundary of $X$ (so the exterior facet polynomials of $f$ have square absolute factors). Thus $\Gamma=C\cdot \partial X$ is non reduced and the computations of residues is much more delicate. In particular, we can not use formula $(10)$. 

Nevertheless, there exists in such a case a (non toric) smooth completion $\wt{X}$ of  $\cp^2$ obtained from $X$ by a serie of blow-ups, and so that the proper transform $\wt{C}\subset \wt{X}$ of $C$ has a transversal intersection with the boundary of $\wt{X}$. By choosing an effective divisor $\wt{D}$ supported on $\partial\wt{X}$ and with sufficiently big multiplicities, we can then use Theorem $1$ efficiently to decompose $\wt{C}$ and to recover the absolute factorization of $f$. The added exceptional divisors give new restrictions on the possible choices of $\wt{\gamma}'<\wt{\gamma}$ (where $\wt{\gamma}=\wt{C}_{|\wt{D}}$) and the presence of singularities of $C$ along $\partial X$  finally turns out to be an opportunity to reduce the recombination number. There remains to find the best choice for $\wt{D}$. This will be explored in a further work.

\subsection{Conclusion}
We propose a new algorithm which computes the absolute factorization of a bivariate polynomial by taking in account the geometry of the Newton polytope. For a sparse polynomial, this permits to reduce the recombination number when compared to the usual vanishing-sums algorithms. There remains to implement such an algorithm and to compute formally its complexity. What we'll remind here is the general idea that
\begin{center}
\textit{The more a curve is singular, the more it is easy to decompose.}
\end{center}
For a sparse polynomial $f$, a naive embedding of the curve of $f$ in $\pp^2$ produces many toric singularities on the line at infinity, and we have shown here that osculation criterions in an appropiate toric resolution $X$ permit to use this sparse information. Moreover, Theorem $1$ permits in theory to profit also of the \textit{non toric} singularities of $C$ on the boundary of $X$.


\end{document}